\documentclass[a4paper,11pt]{article}
\usepackage[utf8]{inputenc}

\usepackage{boxedminipage}
\usepackage{amsfonts}
\usepackage{amsmath} 
\usepackage{amssymb}
\usepackage{graphicx}
\usepackage{amsthm}
\usepackage{t1enc}
\usepackage{subfig}

\newtheorem{theorem}{Theorem}[section]
\newtheorem{lemma}[theorem]{Lemma}

\newtheorem{definition}[theorem]{Definition}

\newtheorem{fact}[theorem]{Fact}
\newtheorem{question}{Question}

\newcommand{\forceP}{\mathbb{P}}
\newcommand{\forceQ}{\mathbb{Q}}
\newcommand{\forceR}{\mathbb{R}}

\newcommand{\forceS}{\mathbb{S}}

\newcommand{\NS}{\hbox{NS}_{\omega_1}}
\newcommand{\ZFC}{\mathsf{ZFC}}

\newcommand{\ZFP}{\mathsf{ZF}^-}

\newcommand{\PFA}{\mathsf{PFA}}
\newcommand{\BPFA}{\mathsf{BPFA}}

\newcommand{\MM}{\mathsf{MM}}
\newcommand{\MMpp}{\mathsf{MM^{++}}}

\newcommand{\BMM}{\mathsf{BMM}}
\newcommand{\MA}{\mathsf{MA_{\omega_1}}}

\newcommand{\Lim}{\mathop{\textrm{Lim}}}

\newcommand{\Cof}{\textrm{Cof}}

\newcommand{\mbp}{\mathbb{P}}

\def\undertilde#1{\mathord{\vtop{\ialign{##\crcr
$\hfil\displaystyle{#1}\hfil$\crcr\noalign{\kern1.5pt\nointerlineskip}
$\hfil\tilde{}\hfil$\crcr\noalign{\kern1.5pt}}}}}

\title{Forcing Axioms and Definabilty of the Nonstationary Ideal on $\omega_1$}
\author{ Stefan Hoffelner\footnote{The first and the third author were supported by funded by the Deutsche Forschungsgemeinschaft (DFG German Research Foundation) under Germanys Excellence Strategy EXC 2044 390685587, Mathematics M\"unster: Dynamics-Geometry-Structure.},  Paul Larson\footnote{The second author was supported in part by NSF research grants DMS-1201494 and DMS-1764320.}, Ralf Schindler and Liuzhen Wu}
\date{5.6.2022}
\begin{document}

\maketitle

\begin{abstract}
We show that under $\BMM$ and ``there exists a Woodin cardinal$"$, the nonstationary ideal on $\omega_1$ can not be defined by a $\Sigma_1$ formula with parameter $A \subset \omega_1$. We show that the same conclusion holds under the assumption of Woodin's $(\ast)$-axiom.
We further show that there are universes where $\BPFA$ holds and $\NS$ is $\Sigma_1(\omega_1)$-definable. Last we show that if the canonical inner model with one Woodin cardinal $M_1$ exists, there is a universe where $\NS$ is saturated, $\Sigma_1(\omega_1)$-definable and $\MA$ holds.
\end{abstract}

\section{Introduction}
This article deals with the possibility of a (boldface) $\bf{\Delta_1}$-definition (over $H(\omega_2)$) of the nonstationary ideal on $\omega_1$ in the presence of various forcing axioms. As we shall see, stronger assumptions rule out the existence of such $\bf{\Sigma}_1$-definitions, whereas weaker assumptions are consistent with such $\bf{\Sigma}_1$-definitions, even in the presence of $\NS$ being saturated.

The main results are as follows.

\begin{theorem}
Assume $\BMM^{}$ and that there exists a Woodin cardinal. Then for no $\Sigma_1$-formula $\varphi(v_0,v_1)$ and no parameter $A \subset \omega_1$ does it hold that
\[ \forall S \in P(\omega_1) (S \text{ is stationary } \Leftrightarrow \varphi(S, A)). \]
\end{theorem}

\begin{theorem}
Assume Woodin's axiom $(\ast)$, then there is no $A \subset \omega_1$ and no $\Sigma_1$-formula $\varphi(-,A)$ in the language of set theory such that \[\forall S \in P(\omega_1) \, (S \text{ is stationary } \Leftrightarrow \varphi(S,A)). \] 
\end{theorem}
In contrast to these two impossibility results we also obtain two theorems which show that under weaker assumptions, $\Sigma_1$ definitions of $\NS$ are possible.

\begin{theorem}
There is a universe in which $\BPFA$ holds and $\NS$ is $\Sigma_1(\omega_1)$-definable.
\end{theorem}

\begin{theorem}
Assume that the canonical inner model with one Woodin cardinal $M_1$ exists. Then there is a generic extension of $M_1$ where
$\NS$ is saturated and $\Sigma_1(\omega_1)$-definable and $\MA$ holds.
\end{theorem}

The paper is organized as follows. We will prove the theorems in the order stated above, thus we start with the two impossibility results, then follow up with the two possibility results. The methods and techniques which are used in this article are quite varied and we will provide only very little preliminary definitions, instead assuming the reader knows the basics of the stationary tower forcing (see \cite{La2} for an extensive account) and $\forceP_{\text{max}}$ (see \cite{La1} or \cite{Wo}), as well as the coding technique of A. Caicedo and B. Velickovic (\cite{CV}).

\section{Impossibility Results}
This section collects two results which show that strong assumptions entail the impossibility of a boldface $\Sigma_1$-definition of $\NS$. We assume that the reader is familiar with Woodin's stationary tower and with $\forceP_{\text{max}}$-forcing.
\subsection{Impossibility under $\BMM^{} +$``There exists a Woodin Cardinal$"$}
The goal of this section is to prove that under $\BMM^{} +$``there exists a Woodin cardinal$"$, no $\Sigma_1$ formula (boldface) can define stationary subsets of $\omega_1$ correctly.

\begin{theorem}
Assume $\BMM$ and the existence of a Woodin cardinal $\delta$. Then for no parameter $A \subset \omega_1$ and for no $\Sigma_1$-formula $\varphi(-,A)$ in the language of set theory,  does $\varphi$ define the stationary subsets of $\omega_1$ i.e, we do not have that \[ \forall T \in P(\omega_1) (T \text{ is stationary } \Leftrightarrow \varphi(T,A) ).\]
\end{theorem}
\begin{proof}

Assume for a contradiction that there is a $\Sigma_1$- formula $\varphi$ and a set $A \subset \omega_1$ such that $\forall T \in P(\omega_1) (T$ is stationary iff $\varphi(T,A)$). Let $\delta$ be our Woodin cardinal.

Let \[ S_0=\{X \prec H(\omega_2) \, : \, |X|=\aleph_1 \land X \text{ is transitive} \} \] and let $g$ be $\forceP_{<\delta}$-generic over $V$ where $\forceP_{<\delta}$ is the associated full stationary tower. 
Let us pick a generic filter $g$ which contains $S_0$, which is possible by the stationarity of $S_0$. As usual we can form the generic elementary embedding in the universe $V[g]$:
\[ j: V \rightarrow M \subset V[g]\]
for a transitive inner model $M$ of $V[g]$. Membership in the generic filter $g$ for the stationary tower forcing can be characterized using $j$, namely
we have that
\[ \forall a (a \in g \leftrightarrow j" \bigcup a \in j(a))\]
thus $S_0 \in g$ yields $j" H(\omega_2)^V \in j(S_0)$.
In particular $j"H(\omega_2)^V$ is transitive, and as $H(\omega_2)^V$ is the transitive collapse of $j"H(\omega_2)^V$, we obtain that $j"H(\omega_2)^V=H(\omega_2)^V$ and that the critical point $crit(j)$ of the elementary embedding $j$ must be $\ge \omega_2^V$.
As $H(\omega_2)^V \in j(S_0)$, $|H(\omega_2)^V|=\aleph_1$ in $M$, so $crit(j)=\omega_2^V$.

We have $P(\omega_1) \cap V \in M$. Indeed by $\BMM$, $2^{\aleph_1}=\aleph_2$, so let $f: \omega_2 \rightarrow P(\omega_1) \cap V$ be a bijection with $f \in V$. Then $j(f) \upharpoonright \omega_2^V =f$ and $j(f) \in  M$ so $P(\omega_1) \cap V = ran(f) \in M$.

It is a theorem of A. D. Taylor (see \cite{Ta}) that $\MA$ implies that $\NS$ is not $\omega_1$-dense. As $P(\omega_1) \cap V$ has size $|\omega_2^V|=\aleph_1$ in $M$, there is a stationary set $D \subset \omega_1$ in $M$ such that $T \backslash D$ is stationary for every stationary $T \in P(\omega_1) \cap V$.
By Theorem 2.5.8 of \cite{La2}, $M^{<\delta} \cap V[g] \subset M$, which gives us that $V[g]$ is a stationary set preserving extension of $V$. Further it is still true in $V[g]$ that $D$ is stationary and $T \backslash D$ is stationary for all $T \notin \NS^{V}$.

In the next step we use the ordinary club shooting forcing $\forceP_{\omega_1 \backslash D}$ over $V[g]$ to shoot a club through the complement of $D$. The forcing will not destroy any stationary subsets from $V \cap P(\omega_1)$:

\begin{fact}
If $h$ denotes a generic filter for $\forceP_{\omega_1 \backslash D}$ over $V[g]$, then if $T \in P(\omega_1) \cap V$ is stationary in $V$ then it will remain stationary in $V[g,h]$.
\end{fact}
\begin{proof}
Fix a stationary (in $V$) $T \in P(\omega_1) \cap V$ and let $p \in\forceP_{\omega_1 \backslash D}$ be a condition and $\tau$ be a name in $V[g]^{\forceP_{\omega_1 \backslash D}}$  such that $p \Vdash \tau$ is a club in $\omega_1$. We shall find a $q' < p$ such that $q' \Vdash \tau \cap T \ne \emptyset$.

As $T\backslash D$ is stationary in $V[g]$, we fix a sufficiently large regular $\theta$ and may pick a countable $X \prec H(\theta)^{V[g]}$ such that $p,\forceP_{\omega_1 \backslash D}$ and $\tau$ are elments of $X$ and which satisfies $\alpha= X \cap \omega_1 \in T \backslash D$. It is straightforward to construct a condition $q \in \forceP_{\omega_1 \backslash D}$, $q<p$ with domain $\alpha$ such that for every dense $D \subset \forceP_{\omega_1 \backslash D}$, $D \in X$ there is a $\xi < \alpha$ such that $q \upharpoonright \xi \in D \cap X$. Finally $q':=q \cup \{ (\alpha, \alpha)\}$ is as desired.

\end{proof}
So $V[g,h]$ is a stationary set preserving extension of $V$.
But now by our hypothesis and by elementarity of $j:V \rightarrow M$ we get that 
\[ M \models \varphi(D,A) \]
and hence
\[ V[g] \models \varphi(D,A) \]
as $\varphi$ is $\Sigma_1$, and consequentially \[V[g,h] \models \varphi(D,A). \]
In $V[g,h]$ the set $D$ is nonstationary, thus
\[V[g,h] \models \exists D (D \text{ is nonstationary } \land \varphi(D,A)). \]
This statement is $\Sigma_1$ with parameter $A \subset \omega_1$ in the language of set theory, and as $\BMM$ is assumed to hold true in $V$ we conclude that 
\[V \models \exists D (D \text{ is nonstationary } \land \varphi(D,A)) \]
which is a contradiction.
\end{proof}

\subsection{Impossibility under $(\ast)$}

Our next goal is to derive the same conclusion from Woodin's $(\ast)$-principle.
Recall that the $(\ast)$-principle states that 
\begin{itemize}
\item $AD$ holds in $L(\mathbb{R})$ and
\item $L(P(\omega_1))$ is a $\forceP_{\text{max}}$-generic extension of $L(\mathbb{R})$.
\end{itemize}  It has been shown very recently by the third author and D. Aspero that $\MMpp$ implies $(\ast)$, solving a long-standing open question. Its proof paved the way for a third impossibility result, namely that under $\MM$, there is no $A \subset \omega_1$ and no $\Sigma_1$-formula which defines stationarity. The proof is due to the third author and 
Xiuyuan Sun and will appear soon (see \cite{RX}).
\begin{theorem}
Assume $(\ast)$, then there is no $A \subset \omega_1$ and no $\Sigma_1$-formula $\varphi(-,A)$ in the language of set theory such that \[\forall T \in P(\omega_1) \, (T \text{ is stationary } \Leftrightarrow \varphi(T,A)). \] 
\end{theorem}
\begin{proof}

Let $V=L(\mathbb{R})[g]$, where $g$ is $\forceP_{\text{max}}$-generic over $L(\mathbb{R})$. Suppose for a contradiction that there is a $\forceP_{\text{max}}$-name for a subset of $\omega_1$, $\dot{A}$, a $\Sigma_1$-formula $\varphi(-,\dot{A})$ and a $\forceP_{\text{max}}$-condition $p= (p,J,b)$ such that
\[ p \Vdash\forall S (S \text{ is stationary iff } \varphi(S, \dot{A})) \]
By \cite{LLS}, the conclusion of the theorem is true for every $A \in P(\omega_1)^{L(\mathbb{R})}$, so we may assume that $\dot{A}^g=A  \in P(\omega_1)^{V} \backslash L(\mathbb{R})$. It is well-known that any such $A \subset \omega_1$ gives rise to a $\forceP_{\text{max}}$-generic filter $g_A:=\{ (p,I,a) \in \forceP_{\text{max}} \, : \,$ $p$ generically iterates to some $p_{\omega_1}=(p_{\omega_1}, \NS^V, A) \}$ and $L(\mathbb{R})[g_A]=L(\mathbb{R})[g]$. Hence we assume without loss of generality that the $\dot{A}$ above is the name for such a distinguished subset of $\omega_1$ which gives rise to the $\forceP_{\text{max}}$-generic filter.

\begin{flushleft}
\textbf{Claim} Let $q < p$, $q=(q,NS^q_{\omega_1},c)$ such that $i: p \rightarrow p^{\ast}=(p^{\ast},J,c)$ witnesses $q<_{\forceP_{\text{max}}} p$ then
 \end{flushleft} 
\[ q \Vdash \exists T \in P(\omega_1) \backslash NS^q_{\omega_1} (\forall S \in (J^{+})^{p^{\ast}} (S \backslash T \notin NS_{\omega_1}^q)).\]

\begin{proof}[proof of the Claim]
Assume that $q < p$ as witnessed by $i: p \rightarrow p^{\ast}$, then
$|p| =\aleph_1$ in $q$ and $q \models \MA$. Hence by the result of Taylor, there is no $\omega_1$-dense, normal ideal in $q$. In particular, there must be a positive set $T \in (\NS^{+})^q$ such that $S \backslash T$ is an element of $(\NS^+)^q$ for all $S \in p^{\ast}, S \in (\NS^+)^q$.

\end{proof}
Now let $q$ be as in the last Claim and for $g$ our $\forceP_{\text{max}}$-generic over $L(\forceR)$ we let $q \in g$ which we can assume by density. By the choice of $p$ and $q<p$, $q \in g$ we have that
\[ V \models S \subset \omega_1 \text{ is stationary } \Leftrightarrow \varphi(S,A).\]
Let $j: q \rightarrow q^{\ast}=(q^{\ast},K,A)$ be the unique iteration of $q$ of length $\omega_1$ given by $g$ which sends $c$ to $A$. By standard $\forceP_{\text{max}}$-arguments, $K=NS_{\omega_1}^V \cap q^{\ast}$.
By elementarity
\[q^{\ast} \models j(T) \in P(\omega_1) \backslash K (\forall S \in j((J^+)^{p^{\ast}} (S \backslash j(T) \notin K))\] and as $\varphi$ is $\Sigma_1$,
\[q^{\ast} \models \varphi(j(T),A). \]
In other words $j(T)$ is stationary in $V$ and for all $S \in j((J^+)^{p^{\ast}})$, $S \backslash j(T)$ is stationary in $V$. Hence
\begin{align*}
(H(\omega_2^V), \NS^V, A) \models & \, j(T) \notin \NS \land \varphi(j(T),A)\\ &\land S \backslash j(T)\notin \NS \text{ for all } S \in j((J^+)^{p^{\ast}}. 
\end{align*}
Now let $L$ denote the ideal generated from $\NS^V \cup \{j(T)\}$, i.e. $X \in L$ iff $(X \backslash j(T)) \cap C = \emptyset$ for some club $C$. It is easy to see that $L$ is a normal ideal and $j((J^+)^{p^{\ast}}) \subset L^+$, so $j(J)=L \cap j(p^{\ast})$.

Now we take a countable elementary substructure of $(H(\omega_2^V), L, A)$ and transitively collapse it. This results in a $\forceP_{\text{max}}$ condition $(r, \bar{L},a)$ (see \cite{Wo}, Lemma 3.14) which gives us:

\begin{flushleft}
\textbf{Claim:} There is a condition $r=(r, \bar{L},a) <_{\forceP_{\text{max}}} p$ such that if $j: p \rightarrow p^{\ast}$ witnesses that $r <_{\forceP_{\text{max}}} p$, then 
\end{flushleft}
\[ r \Vdash \exists T \in \bar{L}  \,  (\varphi(T,a)) \]
But now if we consider the generic iteration of $r$ given by $g$, \[k: (r,\bar{L},a) \rightarrow (r^{\ast}, \NS \cap r^{\ast} ,A)\] of length $\omega_1$, then $r^{\ast} \models \varphi(k(T), A)$ and so $V \models \varphi(k(T), A)$. Yet $k(T)$ is nonstationary in $V=L(\forceR)[g]$. But $r <_{\forceP_{\text{max}}} p$ which is a contradiction.
\end{proof}
\

\section{Possibility Results}
\subsection{$\BPFA$ and the $\Pi_1$-definabilty of $\NS$}
Goal of this section is to show that $\BPFA$ is consistent with a $\Sigma_1$-definition of $\NS^+$. Its proof relies on a new coding technique  which exploits mutually stationary sets.

\subsubsection{Mutually stationary preserving forcing}

\begin{definition}
    Let $K$ be a collection of regular cardinals with bounded below $\kappa$, and suppose that we have $S_\eta\subseteq \eta$ for each $\eta \in K$. Then the collection of sets $\{ S_\eta \mid \eta\in K \}$ is \emph{mutually stationary} if and only if for all algebras $\mathcal{A}$ on $\kappa$, there is an  $N\prec \mathcal{A}$ such that $$       \text{for all }\eta\in K\cap N,\  \sup(N\cap \eta)\in S_\eta.$$
\end{definition}

Foreman-Magidor (\cite{FM}) show that every sequence $\vec S$ with $S_\eta \subset \eta \cap Cof(\omega)$ is mutually stationary. Let $\mathcal{T}_{\vec{S}}$ be the collection of all countable $N$ such that for all $\eta_i\in N$,  $\sup(N\cap \eta_i)\in S_i.$

\begin{theorem}[Foreman-Magidor]\label{thm:fm}
    Let $\langle \eta_i \mid i<j \rangle$ be an increasing sequence of regular cardinals. Let $\vec S=\langle S_i \mid    i<j  \rangle$ be a sequence of stationary sets such that $S_i\subseteq \eta_i \cap \Cof(\omega)$. If $\theta$ is a regular cardinal greater than all $\eta_i$ and $\mathcal{A}$ is a algebra on $\theta$, then there is a $N\prec \mathcal{A}$ belongs to the class $\mathcal{T}_{\vec{S}}$. In particular, $\vec S$ is mutually stationary.
\end{theorem}

From now on, we assume all stationary subsets of ordinals discussed in this section are concentrated on countable cofinality. The corresponding notion for being club in this context is that of an unbounded set which is closed under $\omega$-sequences.

\begin{definition}
    Suppose $\vec S=\{ S_\eta\mid \eta\in K\}$ is mutually stationary and that for every $\eta \in K$, $S_{\eta}$ is stationary, co-stationary in $\eta$ $ \cap $ Cof $ ({\omega})$. We say a forcing poset $\mbp$ is $\vec S$-preserving if the following holds:
    Suppose $\theta> 2^{ \vert \mbp^+\vert}$ is regular. Suppose $M$ is a countable elementary submodel of $H(\theta)$ with $\{\mbp,\vec S\}\subset M$ and $M\in \mathcal{T}_{\vec{S}}$. Suppose $p\in \mbp\cap M$. Then there exists a $(M,\mbp )$-generic condition $q$ extending $p$.
\end{definition}

\textbf{Remark}
    \begin{enumerate}
        \item Any proper forcing is $\vec S$-preserving.
        \item When $K=\{ \omega_1\}$ and $\vec{S}=S \subset \omega_1$, the definition of $\vec S$-preserving is identical to the usual definition of $S$-proper forcing.
        \item Let $\vec{S}$ be such that each $S_{\eta} \in \vec S$ is stationary, co-stationary in $\eta \cap Cof(\omega)$.
       Then an example of a non-proper, $\vec S$ preserving forcing is the forcing poset $Club(S_\eta)$, i.e, the forcing which adds an unbounded subset to $S_\eta$ which is closed under $\omega$-sequences, via countable approximations.
    \end{enumerate}

The iteration theorems for countable support iterated proper forcing can be generalized to $\vec S$-preserving forcing.

\begin{lemma}\label{lem: preservation}
    If $\langle P_i, \dot{Q}_i \mid  i< \alpha \rangle$ is a countable support iterated forcing system and for each $i<\alpha$, $\Vdash_{P_i}$``$\dot{Q}_i$ is $\vec S$-preserving'' then $P_\alpha$ is $\vec S$-preserving.
\end{lemma}
\begin{proof}(Sketch, following the proof of \cite{S}, Theorem 3.2 )
    We will only need to show by induction on $j\le \alpha$ that for any $N\in \mathcal T_{\vec{S}}$, if $j\in N$, then:

    (*)$_N$ For all $i<j$, $i\in N$ and for all $p\in N\cap P_j$ and $q\in P_i$ is $(N,O_i)$ generic condition extending $p\upharpoonright i$, there is an $r\in P_j$ such that $r$ is $(N, P_j)$-generic condition extending $p$ and $r\upharpoonright i=q$.

    The statement (*)$_N$ is identical to the statement (*) in the proof of \cite{S}, Theorem 3.2. It can be checked that the original proof also works here.
\end{proof}

In our proof of the main theorem of this section, we will use forcings which have a specific form, so they get a name.
\begin{definition}
    Let $\kappa$ be an inaccessible cardinal. Let $\vec S= \langle S_{i} \mid i< \kappa \rangle$ be mutually stationary. We say a forcing poset $\mbp$ is a $ \vec S$-coding if $\delta \le \kappa$ and $\mbp= \langle \forceP_\alpha, \dot{\forceQ}_\alpha \mid \alpha<\delta \rangle $ satisfies the following:
    \begin{itemize}
        \item $\mbp$ is a countable support iteration
        \item For any $\alpha<\delta$, one of the followings holds:
        \begin{enumerate}
           
            \item Assume that $\alpha$ is inaccessible and $\mbp_\alpha$ is forcing equivalent to a forcing of size less or equal than $\alpha$.\footnote{We say two forcing $P$ and $Q$ are equivalent if their Boolean completions $B(P)$ and $B(Q)$ are isomorphism.} Assume that in $V^{\mbp_\alpha}$, $\langle B_\beta\mid \beta <2^\alpha\rangle$ is an enumeration of $V_{\alpha}$. Then $\dot{Q}_\alpha$ is allowed to be $$\prod_{j\in B_\beta}Club(S_{\alpha\cdot(\beta+1)+2j}) \times \prod_{j\notin B_\beta}Club(S_{\alpha\cdot(\beta+1)+2j+1})$$ using countable support.
            
             \item In all other cases, we have that $\Vdash_{\mbp_\alpha}\dot{Q}_\alpha$ is proper.
        \end{enumerate}
    \end{itemize}
    Let $\eta$ be an regular cardinal, we  say $\mbp$ is an $\eta$-$\vec S$ coding if (1) is replaced by
    \begin{itemize}
        \item[(1')]  $\alpha \ge \eta$ is inaccessible and $\mbp_\alpha$ is forcing equivalent to a forcing of size less or equal than $\alpha$. In $V^{\mbp_\alpha}$, $\langle B_\beta\mid \beta <2^\alpha\rangle$ is an enumeration of $V_\alpha$. Then $\dot{Q}_\alpha$ is allowed  to be $$\prod_{j\in B_\beta}Club(S_{\alpha\cdot(\beta+1)+2j}) \times \prod_{j\notin B_\beta}Club(S_{\alpha\cdot(\beta+1)+2j+1})$$
    \end{itemize}
\end{definition}
By Lemma \ref{lem: preservation}, if $\mbp$ is a $\vec S$-coding forcing, then $\mbp$ is $\vec S$-preserving. It is also clear that an iteration is a $\vec S$-coding forcing if and only all its initial segments are $\vec S$-coding forcing.

\begin{lemma}\label{lem: stationary preservation}
    Suppose that $\vec S$ is stationary, co-stationary. Suppose $\mbp$ is an $\vec S$ coding forcing. Then for any $i \in \kappa$, the followings are equivalent:
    \begin{itemize}
        \item[(a)] $\Vdash_{\mbp} S_{i}$ contains an $\omega$-club.
        \item[(b)] there are $\beta,\alpha,j$ and a set $B_{\beta} \subset V_{\alpha}$ such that $j \in B_{\beta}$ if $\beta\cdot(\alpha+1)+2j=i$ and $i$ is even and $j \notin B_{\beta}$ if $\beta\cdot(\alpha+1)+2j+1=i$ and  $i$ is odd.
    \end{itemize}
\end{lemma}
\begin{proof} ((b) $\rightarrow$ (a)) Follow from the definition of the forcing.

    ((a) $\rightarrow$ (b)) Fix an $i$ and assume without loss of generality that $i$ is even. Write $i= \alpha \cdot ({\beta+1}) +2j$ and suppose for a contradiction that  $j$ is not an element of $B_{\beta}$. By the definition of an $\vec S$ coding forcing, we must have added a club through $S_{\alpha (\beta+1)+2j+1}$ instead.  Let $\vec T$ be the sequence $\langle T_k \mid k<\kappa \rangle$, where $T_k=S_{k}$ if $k\neq i$ and $T_i=\eta_i\setminus S_{i}$. It follows from Theorem \ref{thm:fm} again that $\vec T$ is mutually stationary. We will prove that $\mbp$ is $\vec T$-preserving to derive a contradiction. Indeed we shall see that $\vec T$-perservation implies that $\eta_i \setminus S_i$ must remain stationary after forcing with $\forceP$, yet $\forceP \Vdash `` S_i$ contains an $\omega$-club$"$ which is impossible.

    To see that $\vec{T}$ preserving forcings preserve the stationarity of every $S_{\eta_i} \in \vec{T}$, we only need to note that for any name $\dot C$ of a subset of $\eta_i$ which is unbounded and $\omega$-closed, and any countable elementary substructure $N$ which contains $\dot C$  and for  which $\sup (N\cap \eta_i)\in S_{\eta_i}$, any $(N,\mbp)$-generic condition $q$, $q$ forces $\dot C\cap ( S_{\eta_i})\neq \emptyset$.

  Next we show by induction that each $Q_{\beta}$ is forced to be $\vec T$-preserving. Work in $V[G_\beta]$. If $\dot Q_{\beta}/G_\beta$ is proper, then it is also $\vec T$-preserving. Otherwise, (1) holds. Now $\dot{Q}_\beta/G_{\beta}$ is a countable support product of shooting club forcing. Fix any $N\in \mathcal T_T$ which is a countable substructure of $H(\theta)^{V[G_\beta]}$. For any $p\in N\cap \dot Q_{\beta}$, we can construct a countable decreasing sequence of conditions $\langle p_i \mid i<\omega \rangle$ meeting all dense set in $N$. Define $q$ by setting $q(i)$ to be the closure of $\bigcup_{n<\omega} p_n(i)$ if $i\in N$ and trivial otherwise. Note any nontrivial $q(i)$ is equal to $\bigcup_{n<\omega} p_n(i)\cup \{\sup(N\cap \eta_i)\}$, where $\eta_i=\sup (S_i)$ is a regular cardinal. As $N\in \mathcal T_{T}$ and no $\beta,\alpha,j$ witness that $\beta\cdot(\alpha+1)+2j=i$ or $\beta\cdot(\alpha+1)+2j+1=i$, we have $\sup(N\cap \eta_{i})\in S_i$, whenever $q(i)$ is non-trivial. Hence $q$ is a condition.
\end{proof}

The proof also show that $\vec S$-coding preserves stationary subset of $\omega_1$ if $ \sup(S_0)>\omega_1$. As a Corollary of Lemma \ref{lem: stationary preservation} and the definition of $\vec S$-coding, in any generic extension by $\vec S$-coding and any even $i$, at most one of $S_{i}$ and $S_{i+1}$ contains a club.
\begin{lemma}\label{lem: composition}
    Suppose $\mbp= \langle P_\alpha, \dot{Q}_\alpha \mid \alpha<\delta \rangle $ is a countable support iteration. Suppose for any $\alpha>0$, $\dot{Q}_\alpha$ is forced to be $\eta_\alpha$-$\vec S$ coding of length $l(\alpha)$, where $\eta_ \alpha=\max \{\vert \mbp_\alpha^+\vert, \Sigma_{\beta<\alpha}l(\beta) \}$. Also let $\eta_0$ be regular. Then $\mbp$ is forcing equivalent to a $\eta_0$-$\vec S$ coding.
\end{lemma}

\begin{proof}
Let $\epsilon_\alpha=\Sigma_{\beta<\alpha}l(\beta)$ for any $\alpha\le\delta$. We define a countable support iteration $\mbp'= \langle P_\alpha',\dot Q_\alpha' \mid \alpha <\epsilon_\delta \rangle$. By induction on $\alpha< \delta$, we give the definition of $P'_{\epsilon_\alpha}$. We will also verify that $P'_{\epsilon_\alpha}$ is forcing equivalent to $P_\alpha$ and $P'_{\epsilon_\alpha}$ is a $\eta_0$-$\vec S$ coding.

If $\alpha$ is limit. Then $P'_{\epsilon_\alpha}$ is the countable support limit of $P'_{\gamma}$ for $\gamma<\epsilon_\alpha$. Hence $P'_{\epsilon_\alpha}$ is the countable support limit of $P'_{\epsilon_\gamma}$ for $\gamma<\alpha$. As $P'_{\epsilon_\gamma}$ is forcing equivalent to $P_{\gamma}$. $P'_{\epsilon_\alpha}$ is forcing equivalent to the countable support limit of $P_{\gamma}$ for $\gamma<\alpha$. Hence $P'_{\epsilon_\gamma}$ is forcing equivalent to $P_{\alpha}$.

If $\alpha=\beta+1$. Let $\epsilon_\alpha=\epsilon_\beta+l(\alpha)$. Now $\dot{Q}_\beta$ is a name for a forcing iteration $\langle P^\alpha_\gamma,\dot{Q}^\alpha_\gamma \mid \gamma<l(\alpha) \rangle$. Now as $P_\beta$ is forcing equivalent to $P'_{\epsilon_\beta}$. We can view $\dot{Q}_\alpha$ as a $P'_{\epsilon_\beta}$ name. By induction on $\gamma< l(\alpha)$, we can set $\dot{Q}'_{\epsilon_\beta+\gamma}$ be $P'_{\epsilon_\beta+\gamma}$ name of $\dot Q^\alpha_\gamma$ and verify that $P'_{\epsilon_\beta+\gamma}$ is forcing equivalent to $P'_{\epsilon_\beta}*P^\alpha_\gamma$.

It remains to check that $P'_{\epsilon_\alpha}$ is a $\eta_0$-$\vec S$ coding. Fix a $\beta<\epsilon_\alpha$, with $\dot{Q}_\beta$ is not forced to be proper. There is $\gamma$ such that $\beta\in [\epsilon_\gamma, \epsilon_{\gamma+1})$. $\beta'$ is inaccessible cardinal greater then $\eta_ \gamma$. Let $\beta= \epsilon_\gamma+\beta'$. As $\epsilon_\gamma< \eta_\gamma$ and $\beta$ is inaccessible. $\beta=\beta'$ is inaccessible. Now $P'_{\beta}$ is forcing equivalent to $P_{\gamma}*P^\gamma_\beta$. $P_{\gamma}$ is of size less than $\beta$ and $P^\gamma_\beta$ is forced to be forcing equivalent to a forcing of size less of equal to $\beta$. Hence $P'_{\beta}$ is also forcing equivalent to a forcing of size less of equal to $\beta$. The rest of (2) is routine to check.
\end{proof}

We are mainly interested in executing $\vec S $-coding forcing over $L$. Now we can define the coding machinery to be used in the latter section.

\begin{definition}
    Suppose $\alpha$ is in inaccessible in $L$ and $X\subseteq P(\alpha)$. We say $\vec S$ codes $X \subset \alpha$ if
    \begin{enumerate}
        \item For any even $i\in [\alpha, (\alpha^+)^L)$, one of $S_{i}$ and $S_{i+1}$ contains a club.
        \item For any $x\in X$,      there is a $\beta< \alpha$ such that $$j\in x\text{ if and only if }S_{\alpha\cdot(\beta+1)+2i}\text{ contains a club}$$ and $$j\notin x\text{ if and only if }S_{\alpha\cdot(\beta+1)+2i+1}\text{ contains a club}$$
    \end{enumerate}
    Let $\vec C=\langle C_i \mid i\in [\alpha, (\alpha^+)^L) \rangle$ be a club sequence which witnesses (2), then we say that $\vec C$ is a $\vec S$ code for $X$.
\end{definition}

A useful fact is the upward absoluteness of the coding  between certain pair of  models.

\begin{lemma}\label{lem: upward}
    Suppose $N\subset M$ are two ZFC$^-$ transitive models. $N\models$ $\vec C$ is a $\vec S$ code for $X \subset \alpha$, $\alpha$ is inaccessible in $L^M$, $((\alpha^+)^L)^M= ((\alpha^+)^L)^N$. Suppose $M\models$ for any even $i\in [\alpha, (\alpha^+)^L)$, at most one of $S_{i}$ and $S_{i+1}$ contains a club. Then $M\models$ $\vec C$ is a $\vec S$ code for $X$.
\end{lemma}
\begin{proof}
    This follows from the definition. Note that being a club is an absolute property between transitive models.
\end{proof}
\begin{lemma}
    Suppose $\mbp$ is a $\vec S$-coding over $L$, $\alpha$ is inaccessible, $X\subseteq P(\alpha)$ and $S$ codes $X$, then $X=P(\alpha)^{L[G_\alpha]}$.
\end{lemma}
\begin{proof}
    By Lemma \ref{lem: stationary preservation}, for any $i$, $S_i$ contains a club if there is some stage $\alpha$ such that $\dot{Q}_\alpha$ satisfies (2) in the definition. But $\dot{Q}_\alpha$ forces $S$ codes $P(\alpha)^{L[G_\alpha]}$ in $L[G_{\alpha+1}]$. Then by Lemma \ref{lem: upward},  $S$ codes $P(\alpha)^{L[G_\alpha]}$ in $L[G]$.
\end{proof}
\subsubsection{ A model of BPFA and $\Delta_1$-definalbility of $NS_{\omega_1}$}

\begin{theorem}
    Suppose that $V=L$ and $\delta$ is a reflecting cardinal. Then there is a forcing poset $P$ such that in $L^{P}$, the following statements hold:
    \begin{enumerate}
        \item BPFA
        \item $\omega_1=\omega_1^L$ and $\omega_2=\delta$.
        \item The nonstationary ideal on $\omega_1$ is $\Sigma_1(\{\omega_1\})$-definable over $\langle H(\omega_2), \in \rangle$.
    \end{enumerate}
\end{theorem}

\begin{proof}
    We first choose a sequence $\langle S_ \alpha\mid \alpha\in \Lim(\delta)\rangle$ uniformly in $\alpha< \delta$ satisfying
    \begin{itemize}
        \item $S_\alpha\subset \alpha$
        \item If $\alpha$ is a regular cardinal, then $S_\alpha$ is stationary co-stationary in $\alpha\cap \Cof(\omega)$.
    \end{itemize}
     The existence of such a sequence $\langle S_ \alpha\mid \alpha\in \Lim(\delta)\rangle$  follows from the fact that $\diamondsuit_{\lambda}$ holds in $L$ for any $L$-cardinal $\lambda$ and is a routine construction. 
     
     Now we define the forcing poset $P$. The forcing $P=\langle P_\alpha, \dot{Q}_\alpha \mid \alpha<\delta \rangle $ will be a countable support iteration of length $\delta$. We require the size of each iterand to be smaller than $\delta$. As a consequence $P$ satisfies $\delta$-cc.
We demand that  $\dot{Q}_\alpha$ is trivial unless $\alpha$ is an inaccessible cardinal and $P_\alpha\subset L_\alpha$. We split into two cases if $\alpha$ is inaccessible.
     \begin{itemize}
        \item ($\alpha$ is Mahlo) We follow an idea from Goldstern-Shelah \cite{GS}.
        Choose a $A\subset \omega_1$ and a $\Sigma_1$ formula $\psi(x,y)$. $A$ and $\psi$ is chosen in a bookkeeping way so that during the whole iteration, each pair $(A,\psi)$ will be dealt with unboundedly many times. Since $\delta$ is reflecting, whenever there is a $\alpha$-$\vec S_{>\alpha}$ coding forcing poset $Q$ which adds a witness to $\psi(x,A)$, there is such a $Q$ of size less than $\delta$. In this case, let $\dot{Q}_\alpha$ be the name of $Q$ and use this forcing at stage $\alpha$. When there is no such $Q$, we set $\dot{Q}_\alpha$ to be trivial forcing.
        \item ($\alpha$ is inaccessible but not Mahlo) $\dot{Q}_\alpha$ is trivial unless $\vert P_\alpha\vert\le \alpha$. Work in $L[G_\alpha]$, where $G_\alpha$ is a $L$-$P_\alpha$ generic filter. $\vert P_\alpha\vert= \alpha$ is inaccessible. We choose $\langle B_\beta\mid \beta <2^\alpha\rangle$ to be an enumeration of $P(\alpha)^{V[G_\alpha]}$. Let $Q$ be the forcing $$\prod_{i\in B_\beta}Club(S_{\alpha\cdot(\beta+1)+2i}) \times \prod_{i\notin B_\beta}Club(S_{\alpha\cdot(\beta+1)+2i+1})$$ with countable support. Then we force with $Q$ at stage $\alpha$.
     \end{itemize}
     It follows from the definition of $P$ that for any $\alpha <\delta$, the forcing $\dot{Q}_{\alpha}$ is forced to be a $\alpha-\vec S$ coding forcing. Applying Lemma \ref{lem: composition}, we know that $P$ is a $\vec{S}$-coding forcing. Moreover, for any $\alpha$ Mahlo, the tail $P /P_\alpha$ is a $\alpha$-$\vec S_{>\alpha}$ coding.

     Now as $P$ is $\vec S$-preserving, $P$ preserves $\omega_1$. On the other hand, $P$ is $\delta$-c.c and $P$ preserves cardinals above $\delta$. By the definition of the forcing in the inaccessible but not Mahlo case, all cardinals below $\delta$ are collapsed to $\omega_1$. In summary, $\omega_1^{L^P}=\omega_1^L$ and $\omega_2^{L^P}=\delta$.

     \begin{lemma}
        $P \Vdash $ $\BPFA$.
     \end{lemma}

     \begin{proof}
        Work in $L[G]$. Let $A\subset \omega_1$ and $Q \in L[G]$ be a proper forcing which adds a witness to the $\Sigma_1$-formula $\psi(x,A)$. Now let $\alpha$ be a stage such that $A\in L[G_\alpha]$, $\alpha$ is Mahlo and $( A,\psi)$ is to be dealt with at stage $\alpha$. In $L[G]$, the forcing $P/P_\alpha*Q$ is a $\alpha-\vec S_{>\alpha}$ coding forcing adding a witness to $\psi(x,A)$. Hence, we must be in the first case of the definition of our iteration at stage $\alpha$ and $Q_\alpha$ is a $\alpha-\vec S_{>\alpha}$ coding forcing which adds a witness to $\psi(x,A)$. Thus $H(\omega_2)^{L[G_{\alpha+1}]}\models \exists x\psi(x, A)$. By upward absoluteness, $H(\omega_2)^{L[G]}\models \exists x\psi(x, A)$.
     \end{proof}
     \begin{lemma}\label{lem: stat des}
      Work in $L[G]$.  Let $S$ be a subset of $\omega_1$, then the following are equivalent:
        \begin{itemize}
            \item[(a)] $S$ is stationary.
            \item[(b)] there is a $\alpha$ inaccessible in $L$ with $P_\alpha\subseteq L_\alpha$ and $S$ is stationary in $L[G_\alpha]$.
            \item[(c)] there is a $\alpha$ inaccessible in $L$, there is $\vec{C} \in L[G_{\alpha+1}]$ which is a $\vec S$ code for  $P(\alpha)^{L[G_\alpha]}$, and there is a transitive model $M$ of $\ZFP$ such that
            $\vec{C} \in M$ and $M$ thinks that $S$ is stationary.
        \end{itemize}

     \end{lemma}
     \begin{proof}
         ((a) $\rightarrow$ (c)) Let $\alpha$ be inaccessible but not Mahlo and consider stage $\alpha$ of the iteration which we can assume to be  such that $S\in L[G_\alpha]$ and $\vert P_\alpha\vert\le \alpha$. Now $Q_\alpha$ forces the existence of a $\vec{C}$-sequence which is an $\vec S$ code for $P(\alpha)^{L[G_\alpha]}$.   Any transitive $M$ which contains $\vec{C}$ is as desired in (c), moreover $M$ will automatically think that $S$ is stationary if $S$ is stationary in $L[G]$.

         ((b) $\rightarrow$ (a)) The tail of forcing $P/P_\alpha$ is $\alpha$-$S_{>\alpha}$ coding. Now the proof of Lemma \ref{lem: stationary preservation} shows $P/P_\alpha$ preserves stationarity of $S$.

         ((c) $\rightarrow$ (b)) If $M$ is as in the assumption, $P(\alpha)^{L[G_{\alpha}]}$ is a subset of $M$, and if $M$ thinks that $S$ is stationary, so must $V_{\alpha+1}^{L[G_{\alpha}]}=P(\alpha)^{L[G_{\alpha}]}$, and hence $L[G_{\alpha}]$.
     \end{proof}

     We now present the $\Sigma_1$ definition of stationary subsets of $\omega_1$ over $(H(\omega_2), \in, \{\omega_1\})$. Let $\psi(x)$ describe the following statement: there are objects $A$ and $M$ such that
     \begin{itemize}
         \item[(i)] $A= \langle C_i \mid i
         \in [\alpha, \beta) \rangle$ is a sequence with $C_i$ is an $\omega$-club in $\sup (C_i \cap Ord)=\eta_i$.
         \item[(ii)] All $\eta_i$ are regular cardinals in $L$. $\eta_\alpha=\alpha$ is inaccessible in $L$.
         \item[(iii)] $M$ is a transitive ZFC$^-$ model with $\omega_1=\omega_1^M$, $((\alpha^+)^L)^M= (\alpha^+)^L$, and $x\in M$.
         \item[(iv)] $M\models x \subset \omega_1$ is stationary.
         \item[(v)] $M\models A$ is a $\vec S$ codes for $P(\alpha)$.
     \end{itemize}
     It is routine to check (i), (iii), (iv) and (v) are all $\Sigma_1$ over $\langle H(\omega_2),\in,  \{\omega_1\} \rangle$. For (ii), as $L[G]\models BPFA$, we can apply a trick of Todorcevic (cf. Proof of Lemma 4, \cite{To02}) to get $\Sigma_1$ formulas $\psi_0 (x)$ and $\psi_1 (x)$ such that $H(\omega_2)\models \psi_0(\beta)$ if and only if $\beta$ is a regular cardinal in $L$ and $H(\omega_2)\models \psi_1(\beta)$ if and only if $\beta$ is an inaccessible cardinal in $L$. 
     
 Here $\psi_0(x)$ is the formula describing the existence of a specialization function of the tree $T_x$, where $T_x$ is derived from the canonical global square sequence in $L$. It is a consequence of $\BPFA$ that $x$ is uncountable regular in $L$ if and only if such a specialization function exists. Now $\psi_1(x)$ is a formula says $\alpha$ is regular and a limit of regular cardinals in $L$.

    It is now clear that $\psi(x)$ is $\Sigma_1$ over $\langle H(\omega_2),\in, \{\omega_1\} \rangle.$ What is left is to show that $\psi(x)$ indeed characterizes stationarity.

     \begin{lemma}
         In $L[G]$, for any $x\subseteq \omega_1$, $H(\omega_2)\models\psi(x)$ if and only if $x$ is stationary.
     \end{lemma}
     \begin{proof}
       If $x\subset \omega_1$ is stationary, then, by Lemma \ref{lem: stat des} (c), we do immediately get a witness $M \in H(\omega_2)$ and an $A$ so that $\psi(x)$ holds.
        
        Conversely, let $x$, $A$, $M$ be as given by the definition of $\psi(x)$, we check that $x$ is stationary. But if  $\psi(x)$ holds, then $M$ thinks that $A$ is a $\vec{S}$-code for $P(\alpha)$. By item  $(ii)$ of $\psi$, this also means that $A$ is a $\vec{S}$-code for $P(\alpha)$ and by (c) of Lemma \ref{lem: stat des} it is true that $x$ is indeed stationary. 
     \end{proof}

\end{proof}

\subsection{A model for $\MA$, $\NS$ being saturated and $\Delta_1(\{\omega_1\})$-definable}

In this section we improve an earlier result of \cite{Ho} namely
we show that given a Woodin cardinal, there is a model such that $\NS$ is saturated, $\Delta_1$-definable with parameters from $H(\omega_2)$ while Martin's Axiom also holds true. If one forces over the canonical inner model with one Woodin cardinal $M_1$, then the construction yields a model where additionally $\NS$ is definable with $\omega_1$ as the only parameter. 

\subsubsection{Short summary of the main features of the model $W_1$}
The proof relies heavily on the coding machinery introduced in \cite{Ho}, where it is shown that, given a Woodin cardinal $\delta$, then there is a universe such that $\NS$ is saturated and $\Delta_1(\vec{C}, \vec{T}^0)$-definable, where $\vec{C}$ is an arbitrary ladder system on $\omega_1$ and $\vec{T}^0$ is an independent sequence (independence will be defined in a moment below) of Suslin trees of length $\omega$.
As this coding is rather convoluted, we will not define it here in detail but instead only highlight the most important notions and features of it. Our notation will be exactly as there.

We shall use Suslin trees on $\omega_1$ for creating a $\bf{\Sigma_1}$-definition of stationarity. To facilitate things, the trees should satisfy a certain property:
\begin{definition}
 Let $\vec{T} = (T_{\alpha} \, : \, \alpha < \kappa)$ be a sequence of Suslin trees. We say that the sequence is an 
independent family of Suslin trees if for every finite set $e= \{e_0, e_1,...,e_n\} \subset \kappa$ the product $T_{e_0} \times T_{e_1} \times \cdot \cdot \cdot \times T_{e_n}$ 
 is a Suslin tree again.
\end{definition}

Independent sequences of Suslin trees can be used to code arbitrary information using the two well-known and mutually exclusive ways to destroy a Suslin tree, namely either shooting a branch through the tree or specializing it.
More precisely, given a set $X \subset \omega_1$ and an independent sequence of Suslin trees $\vec{T}=(T_i \, : \, i < \omega_1)$, we can code the characteristic function of $X$ into $\vec{T}$ via forcing with the finitely supported product of 
\begin{equation*}
\forceP_i= \begin{cases}
T_i  & \text{ if } i \in X \\ Sp(T_i) & \text{ if } i \notin X
\end{cases}
\end{equation*}
where $T_i$ just denotes the forcing one obtains when forcing with the Suslin tree $T_i$ (which adds a cofinal branch to $T_i$) and $Sp(T_i)$ denotes the forcing which spezializes the tree $T_i$.
Note that the independence of $\vec{T}$ buys us that the finitely supported product has the ccc as well. We will eventually use this mechanism to create a generic extension of $V$ with a $\Sigma_1(\vec{C},\vec{T}^0)$-definition of being stationary on $\omega_1$.

The construction of such a universe shall be sketched now.
We start with a universe $V$ with one Woodin cardinal $\delta$ with $\diamondsuit$. We fix a ladder system $\vec{C}$ and an $\omega$-sequence of independent Suslin trees $\vec{T}^0$ and start a first, nicely supported iteration (using Miyamoto's nice iterations, see \cite{Mi}) of length $\delta$ over $V$ which combines Shelah's proof of the saturation of $\NS$ from a Woodin cardinal with the forcings invented by A. Caicedo and B. Velickovic (see \cite{CV}) and other forcings whose purpose is to create a model $W_0$ which will have several features listed below which will turn our to be useful.  

In a next step we force over $W_0$ with a variant of almost disjoint coding which is due to L. Harrington (see \cite{Ha}), which is used to ensure that over the resulting generic extension of $W_0$, denoted by $W_1$, there is a ${\Sigma_1}(\vec{C},\vec{T}^0)$-definable $\omega_2$-sequence of $\aleph_1$-sized, transitive models which are sufficiently smart to determine whether a member is a stationary subset of $\omega_1$ or a Suslin tree in $W_0$. These sufficiently smart models are called suitable, and the purpose of suitable models is that they can correctly compute a fixed, independent sequence $\vec{T}=(T_{\alpha} \, : \, \alpha < \omega_2)$ of Suslin trees. Due to this correctness, we can identify whether an $\omega_1$-block of trees is in $\vec{T}$ in a $\Sigma_1(\vec{C},\vec{T}^0)$-way. The sequence $\vec{T}$ will be used later to code up being stationary in a $\Sigma_1(\vec{T}^0, \vec{C})$-way.  

To summarize the above, starting from an arbitrary $V$ which contains a Woodin cardinal $\delta$ with $\diamondsuit$, and fixing a ladder system on $\omega_1$ $\vec{C}$ and an $\omega$-sequence of independent Suslin trees $\vec{T}^0$, we create first a generic extension $W_0$ and then a further generic extension $W_1$ such that in $W_1$ the following holds:

\begin{enumerate}
\item $\delta=\aleph_2$.
\item In $W_1$, the nonstationary ideal is saturated and the saturation is ccc-indestructible.
\item Every real in $W_0$ is coded by a triple of limit ordinals $(\alpha, \beta, \gamma)$ below $\omega_2$ relative to the ladder system $\vec {C}$ (in the sense of Caicedo-Velickovic, see \cite{Ho}, pp. Theorem 18 $(\ddagger)$).
\item Every subset $X \subset \omega_1$, $X \in W_0$ is coded by a real $r_X \in W_0$ relative to the fixed almost disjoint family of reals $F$ we recursively obtain from our ladder system $\vec{C}$.
\item There is an $\omega_2$-sequence of independent Suslin trees $\vec{T} = (T_i \, : \, \omega < i < \omega_2) \in W_1$ whose initial segments are uniformly and correctly definable in suitable models. The set of suitable models is itself $\Sigma_1(\vec{C}, \vec{T}^0)$-definable in $W_1$ using as parameters the ladder system $\vec{C}$, and one $\omega$-block of independent Suslin tree $\vec{T}^0=(T_n \, : \, n \in \omega)$.
As a consequence, $\vec{T}$ is $\Sigma_1(\vec{C},\vec{T}^0)$-definable over $W_1$.

\item The definition of $\vec{T}$ remains valid in all generic extensions of $W_1$ by forcings with the countable chain condition. So $W_1$ is a reasonable candidate for a ground model using coding forcings which themselves have the ccc. 
\end{enumerate}

In a second iteration using $W_1$ as the ground model,  we force with coding forcings using Suslin trees that are applied to make $\NS$ $\Sigma_1(\vec{C},\vec{T}^0)$-definable. 
We force with a finitely supported iteration of ccc forcings over $W_1$. As a consequence we preserve the saturation of the nonstationay ideal and the sequence $\vec{T}$ is still $\Sigma_1(\vec{C}, \vec{T}^0)$-definable. The only forcings which are used in this second iteration are the Suslin trees from our independent sequence $\vec{T}^{>0}=(T_i \, : \, i > \omega)$, which we either specialize or destroy via the addition of an $\omega_1$-branch. We will use a bookkeeping function and start to write characteristic functions of every stationary subset of $\omega_1$ into $\omega_1$-blocks of $\vec{T}^{>0}$ using either the spezialization forcing or shooting a branch through elements of $\vec{T}^{>0}$.

This will eventually yield a universe $W_{\omega_2}$ where the nonstationary ideal remains saturated and where stationary subsets of $\omega_1$ can be characterized as follows:

\begin{fact}
There are $\Sigma_1(\vec{C}, \vec{T}^0)$-formulas $\Phi(r)$ and $\Psi(S)$ where the formula $\Phi(r)$ defines a set of reals such that every member is an almost disjoint code for an $\aleph_1$-sized, transitive models which can be used to compute the sequence $\vec{T}^{>0}$ of Suslin trees correctly (these models are the suitable models mentioned earlier). The formula $\Psi(S)$ then defines stationary subsets of $W_{\omega_2}$ in the following way:
\begin{itemize}
\item[$\Psi(S)$] if and only if there is an $\aleph_1$-sized, transitive model $N$ which contains $\vec{C}$ and $\vec{T}^0$ such that $N$ models that
\begin{itemize}
\item There exists a real $x$ such that $\Phi(x)$ holds, i.e. $x$ is a code for a suitable model $M$.
\item There exists an ordinal $\alpha$ in the suitable model $M$ such that $\vec{T}'$ is the $\alpha$-th $\omega_1$ block of the definable sequence of independent Suslin trees as computed in $M$ and $N$ sees a full pattern on $\vec{T}'$.
\item $\forall \beta < \omega_1$ $( \beta \in S$ if and only if $\vec{T}' (\beta)$ has a branch).
\item $\forall \beta < \omega_1$ $(\beta \notin S$ if and only if $\vec{T}' (\beta)$ is special).
\end{itemize}
\end{itemize}
Note that $\Psi(S)$ is of the form $\exists N (N \models...)$, thus $\Psi$ is a $\Sigma_1$-formula.
\end{fact}

\subsubsection{Forcing over $W_1$}
We shall show how to modify the just sketched construction to obtain a model where additionally $\MA$ holds. As before we will construct the model $W_1$. We will proceed however not coding up stationary subsets of $\omega_1$ as we do in \cite{Ho} but instead code up a second $\omega_2$-block of generically added Suslin trees first.

In an $\omega_2$-length iteration we first use
 Tennenbaum's forcing over $W_1$ to add an independent sequence of Suslin trees of length $\omega_2$ and use our definable independent sequence $\vec{T}$ to code up the added Suslin trees. First let us briefly recall the definition of Tennenbaum's forcing.
\begin{definition}
Tennenbaum's forcing $\forceP_T$ consists of conditions which are finite trees $(T,<_T)$, $T \subset \omega_1$ such that $\alpha < \beta$ if $\alpha <_T \beta$, and $(T_1, <_{T_1}) < (T_2,<_{T_2})$ holds if $T_2 \subset T_1$ and $<_{T_2} = <_{T_1} \cap (T_2 \times T_2)$.
\end{definition}
It is well-known that $\forceP_T$ is Knaster and adds generically a Suslin tree to the ground model.

So we start with the model $W_1$ as our ground model. Let $\vec{C}$ be our fixed ladder system on $\omega_1$ and let $\vec{T}^0$ be a fixed independent sequence of Suslin trees of length $\omega$. In $W_1$ there is a $\Sigma_1(\vec{C}, \vec{T}^0)$-definable $\omega_2$-sequence of $\omega_1$-blocks of independent Suslin trees $\vec{T}=(\vec{T}^{\alpha} \, : \, \alpha < \omega_2)$, where for every $0 \ne \alpha < \omega_2$, $\vec{T}^{\alpha}=(T_{\eta}^{\alpha} \, : \, \eta < \omega_1)$, and $\vec{T}$ forms an independent sequence of Suslin trees.

Over $W_1$ we start a finitely supported iteration $\forceQ= ((\forceQ_{\alpha},\dot{\forceR}_{\alpha}) \, : \, \alpha < \omega_2)$ and let $H_{\alpha}$ denote the generic filter for $\forceQ_{\alpha}$. For every $\alpha< \omega_2$, using $W_1[G_{\alpha}]$ as the ground model, $\dot{\forceR}^{G_{\alpha}}_{\alpha}$ is defined to be $\forceQ^1_{\alpha} \ast \forceQ^2_{\alpha}$, and $\forceQ^1_{\alpha}$ is Tennenbaum's $\forceP_T$ and $\forceQ_{\alpha}^2$ codes up the generically by $\forceQ_{\alpha}^1$ added tree, called $h_{\alpha} \subset \omega_1$   in the following way:

\begin{equation*}
\forceQ_{\alpha}^2 = \begin{cases}
T^{\alpha}_{\eta}  & \text{ if } \eta \in h_{\alpha} \\ Sp(T^{\alpha}_{\eta}) & \text{ if } \eta \notin h_{\alpha}
\end{cases}
\end{equation*}

It is immediate that the resulting universe $W_2= W_1[H_{\omega_2}]$ is a ccc extension of $W_1$, thus $\NS$ remains saturated and as in \cite{H} one can show that the generically added sequence of Suslin trees $(h_{\alpha} \, : \, \alpha < \omega_2)$ is an independent, $\Sigma_1(\vec{C}, \vec{T}^0)$-definable sequence of Suslin trees via the formula:

\begin{fact}
There is a $\Sigma_1(\vec{C}, \vec{T}^0)$-formula $\Psi(h)$ which defines the generically added Suslin trees $(h_{\alpha} \, : \, \alpha < \omega_2)$ of $W_2$:
\begin{itemize}
\item[$\Psi(h)$] if and only if there is an $\aleph_1$-sized, transitive model $N$ which contains $\vec{C}$ and $\vec{T}^0$ such that $N$ models that
\begin{itemize}
\item There exists a real $x$ such that $\Phi(x)$ holds, i.e. $x$ is a code for a suitable model $M$.
\item There exists an ordinal $\alpha$ in the suitable model $M$ such that $\vec{T}'$ is the $\alpha$-th $\omega_1$ block of the definable sequence of independent Suslin trees as computed in $M$ and $N$ sees a full pattern on $\vec{T}'$.
\item $\forall \beta < \omega_1$ $( \beta \in h$ if and only if $\vec{T}' (\beta)$ has a branch).
\item $\forall \beta < \omega_1$ $(\beta \notin h$ if and only if $\vec{T}' (\beta)$ is special).
\end{itemize}
\end{itemize}
Note that $\Psi(h)$ is of the form $\exists N (N \models...)$, thus $\Psi$ is a $\Sigma_1$-formula.
\end{fact}

In a second step, we use $W_2$ as our ground model and force in an $\omega_2$-length, finitely supported iteration $\MA$ while simultaneously coding up stationary subsets of $\omega_1$ using the boldface $\Sigma_1$-definable sequence of trees $(h_{\alpha} \, : \, \alpha < \omega_2)$. We write $\vec{h}^{\alpha}$ for the $\alpha$-th $\omega_1$-block of elements of the sequence $(h_{\beta} \, : \, \beta < \omega_2)$ and let $h^{\alpha}_{\beta}$ denote the $\beta$-th element of the $\alpha$-th $\omega_1$-block.

We do the usual forcing to code characteristic functions of stationary subsets of $\omega_1$ into $\vec{h}^{\alpha}$, but additionally we feed in forcings of size $\aleph_1$ with the countable chain condition to produce a model of $\MA$. Note that as all the forcings we use have the ccc and thus we will preserve the saturation of $\NS$.

In order to prevent the forcings we use to get $\MA$ from damaging our codes, we will force $\MA$ in a ``diagonal way$"$. 
We define a finitely supported iteration $((\forceR_{\alpha}, \dot{\mathbb{S}}_{\alpha} ) \, : \, \alpha < \omega_2)$ of ccc forcing over $W_2$ inductively using a bookkeeping function $F$. We let $F \in W_2$, $F: \omega_2 \rightarrow \omega_2 \times \omega_2 \times 2$ such that for every $(\alpha, \beta, i) \in \omega_2 \times \omega_2 \times 2$, $F^{-1} (\alpha, \beta, i)$ is an unbounded subset in $\omega_2 \times \omega_2 \times 2$.  Assume we are at stage $\alpha < \omega_2$ and we have already defined the iteration $\forceR_{\alpha}$ up to $\alpha< \omega_2$.  We let $I_{\alpha}$ denote the generic filter for $\forceR_{\alpha}$. We also assume by induction that every of the first $\alpha$-many $\omega_1$-blocks of trees $\vec{h}^{\beta}$, $\beta < \alpha$ has already been used for codings, but the elements of the sequences $(\vec{h}^{\eta} \, : \eta \ge \alpha)$ still form an independent sequence of Suslin trees in $W_2[I_{\alpha}]$. The forcing we have to use next is determined by the value of $F(\alpha)$. 
\begin{enumerate}
\item If $F(\alpha)=(\beta, \gamma, 0)$, then we look at the $\beta$-th stationary subset $S$ of $\omega_1$ in $W_2[I_{\gamma}]$ and use the $\alpha$-th $\omega_1$-block of $\vec{h}$, $\vec{h}^{\alpha}=(h^{\alpha}_{\eta} \, : \, \eta < \omega_1)$ to code up $S$ i.e. we will force with
$\dot{\forceS}^{I_{\alpha}}_{\alpha} := \prod_{i < \omega_1} \forceP_i$ with finite support, where
\begin{equation*}
\forceP_i = \begin{cases}
h^{\alpha}_{\eta}  & \text{ if } \eta \in S \\ Sp(h^{\alpha}_{\eta}) & \text{ if } \eta \notin S
\end{cases}
\end{equation*}
where $h^{\alpha}_{\eta}$ here is considered as a forcing notion when forcing with the tree and $Sp(h^{\alpha}_{\eta})$ denotes the specialization forcing for the tree $h^{\alpha}_{\eta}$.
\item If $F(\alpha)=(\beta, \gamma, 1)$, then we look at the $\beta$-th forcing $\mathbb{B}$ of size $\aleph_1$ in $W_2[I_{\gamma}]$ which has the countable chain condition as seen in the universe $W_2[I_{\alpha}]$. We can consider the iteration $((\forceR_{\zeta} \, : \, \zeta < \alpha) \ast \mathbb{B})$ which has a dense subforcing of size $\aleph_1$ (the dense set is just the set of conditions in $\forceR_{\alpha} \ast \mathbb{B}$ which are fully decided), ccc forcing in $W_2$ and can thus be seen as a subset of $\omega_1$ in $W_2$. As $W_2=W_1[H_{\omega_2}]$, there is a stage $\nu < \omega_2$ such that (a forcing equialent to) $((\forceR_{\zeta} \, : \, \zeta < \alpha) \ast \mathbb{B}) \in W_1[H_{\nu}]$. Now if $\nu \le \alpha$ we let $\dot{\forceS}^{I_{\alpha}}_{\alpha}$ be $\mathbb{B}$. Otherwise we force with the trivial forcing.
\end{enumerate}

In the first case of the definiton we will write codes in the sequence $(\vec{h}^{\alpha} \, : \, \alpha < \omega_2)$ of blocks of Suslin trees. The definition of the second case ensures that we will not accidentally write an unwanted pattern when forcing for $\MA$.
\begin{lemma}
Assume we are at stage $\alpha$ of our iteration and we are in the nontrivial part of case 2 of the definition, thus we force with an $\aleph_1$-sized $\dot{\forceS}^{I_{\alpha}}_{\alpha}=\mathbb{B}$. Then, if $I_{\alpha+1}$ is $\forceR_{\alpha+1}=\forceR_{\alpha} \ast \dot{\forceS}_{\alpha}$-generic over $W[H_{\omega_2}]$, all the Suslin trees in $\vec{h}_{\zeta}$, $\zeta \ge \alpha$ remain Suslin trees in $W[H_{\omega_2}][ I_{\alpha+1}]$.
\end{lemma}
\begin{proof}
Assume that we are at stage $\alpha$, thus the model we have produced so far is $W_1[H_{\omega_2}][ I_{\alpha}]$ and we force with $\dot{\forceS}_{\alpha}$ which is a ccc forcing of size $\aleph_1$ in $W_1[H_{\nu}][I_{\alpha}]$ for $\nu \le \alpha$. Consider some block $\vec{h}^{\zeta}$, $\zeta \ge \alpha$ in the universe $W_1[H_{\omega_2}][\forceR_{\alpha} \ast \dot{\forceS}_{\alpha}]$. The latter universe is obtained via
the iteration 
$$(\forceQ_{\zeta} \, : \, \zeta < \omega_2) \ast \forceR_{\alpha} \ast \dot{\forceS}_{\alpha}= (\forceQ_{\zeta} \, : \, \zeta \le \nu) \ast  (\forceQ_{\zeta} \slash \forceQ_{\nu} \, : \zeta> \nu) \ast \forceR_{\alpha} \ast \dot{\forceS}_{\alpha} $$ 
and the right hand side can be rewritten as $$(\forceQ_{\zeta} \, : \, \zeta \le \nu)  \ast  ((\forceQ_{\zeta} \slash \forceQ_{\nu} \, : \zeta> \nu) \times (\forceR_{\alpha} \ast \dot{\forceS}_{\alpha} )).$$ As we can switch the order in products, the latter can be written as $$(\forceQ_{\zeta} \, : \, \zeta \le \nu)  \ast  ((\forceR_{\alpha} \ast \dot{\forceS}_{\alpha} )\times (\forceQ_{\zeta} \slash \forceQ_{\nu} \, : \zeta> \nu)),$$ and consequentially the $W_1[H_{\nu}]$-generic filter $H_{\nu, \omega_2}$ for the tail $(\forceQ_{\zeta} \slash \forceQ_{\nu} \, : \, \zeta> \nu, \zeta < \omega_2)$ remains generic over the model $W_1[(\forceQ_{\zeta} \, : \, \zeta \le \nu)  \ast  (\forceR_{\alpha} \ast \dot{\forceS}_{\alpha})]$.  This means in particular that the generically added Suslin trees $\vec{h}^{\zeta}, \zeta > \nu$ for Tennenbaum's forcing for adding a Suslin tree which are elements in $H_{\nu, \omega_2}$ remain generic even over the ground model $W_1[H_{\nu}  \ast  I_{\alpha+1}]$. Now Tennebaum's forcing is computed in an absolute way in every universe with the same $\omega_1$, and trivially, every generic filter for it is a Suslin tree. Hence we obtain that every $h^{\zeta}_{\eta}$, $\zeta > \nu, \eta< \omega_1$ is a Suslin tree in $W_1[H_{\nu} \ast  I_{\alpha+1} ]$, thus every $h^{\zeta}_{\eta}$, $\zeta \ge \alpha \ge \nu$ is a Suslin tree in $W_1[H_{\omega_2} \ast I_{\alpha+1}]$ as claimed.
\end{proof}
After $\omega_2$-many steps we arrive at $W_2[I_{\omega_2}]$ which has the desired properties. The first thing to note is that we are in full control of the codes which are written into the sequences of blocks of Suslin trees $(\vec{h}^{\alpha} \, : \, \alpha < \omega_2)$.
\begin{lemma}
In $W_2[I_{\omega_2}]$, a set $S \subset \omega_1$ is stationary if and only if there is an $\alpha < \omega_2$ such that 
\begin{itemize}
\item[] $\forall \beta < \omega_1 (( \beta \in S \leftrightarrow h^{\alpha}_{\beta}$ has a branch $)$ and $(\beta \notin S \leftrightarrow h^{\alpha}_{\beta}$ is special $))$.
\end{itemize}
\end{lemma}
\begin{proof}
If $S$ is stationary then the rules of the iteration guarantee that there is such an $\alpha < \omega_2$ with the desired properties. On the other hand if there is an $\alpha < \omega_2$ such that the $\alpha$-th block $\vec{h}^{\alpha}$ sees a certain 0,1-pattern then by the last Lemma, this pattern must come from the first case in the definition of our iteration. Hence $S$ has to be stationary.
\end{proof}

\begin{theorem}
In $W_2[I_{\omega_2}]$ $\MA$ holds, $\NS$ is saturated and $\Delta_1(\vec{C}, \vec{T}^0)$-definable.
\end{theorem}
\begin{proof}
That $\NS$ is saturated is clear as $W_2[I_{\omega_2}]$ is a ccc extension of $W_0$ and $\NS$ is saturated in $W_0$ and ccc indestructible. The proof that $\MA$ holds in $W_2[I_{\omega_2}]$ is also clear as a standard computation yields that the continuum is $\aleph_2$ in $W_2[I_{\omega_2}]$. Hence, it is sufficient to show that $\MA$ holds for ccc posets of size $\aleph_1$. Let $\forceP \in W_2[I_{\omega_2}]$ be such, then there is a stage $\nu < \omega_2$ such that $\forceP \in W_2[I_{\nu}]$. The rules of the iteration yield that we will consider $\forceP$ unboundedly often after stage $\nu$. Thus there will be a stage $\alpha< \omega_2$ such that $\forceP$ is considered by the bookkeeping $F$ and $\forceR_{\alpha} \ast \forceP$  is an element of $W_1[H_{\nu}]$ for $\nu \le \alpha$, and hence we used $\forceP$ in the iteration $(\forceR_{\eta} \, : \, \eta < \omega_2)$, so $\MA$ holds. 

In order to see that every $S \in W_2[I_{\omega_2}]$ has a $\Sigma_1(\vec{C}, \vec{T}^0)$-definition we exploit the fact that the trees $(h_{\alpha} \, : \, \alpha < \omega_2)$ are $\Sigma_1(\vec{C}, \vec{T}^0)$-definable in $W_1[I_{\omega_2}]$. We claim that the following $\Sigma_1(\vec{C}, \vec{T}^0)$-formula $\varphi(S)$ defines being stationary in $W_2[I_{\omega_2}]$:
\begin{itemize}
\item[] $\varphi(S)$ if and only if there exists a triple of $(M_1, M_2, M_3)$ of transitive models of size $\aleph_1$ such that $M_1 \subset M_2 \subset M_3$, $M_1$ is a suitable model and $M_2$ sees a full pattern on the trees in some $\omega_1$-block $\vec{T}^{\alpha}$. This pattern itself yields an $\omega_1$-block of trees $\vec{h}^{\beta}$ and $M_3$ sees a full pattern on $\vec{h}^{\beta}$ and this pattern is the characteristic function for $S$.
\end{itemize}
Note that the formula $\varphi(S)$ is of the form $\exists M_1,M_2,M_3 \,\sigma(M_1,M_2,M_3,S)$ and $\sigma$ is $\Delta_1$ as all the statements in $\sigma$ are of the form $M_i \models...$ which is a $\Delta_1$-formula.

By absoluteness and the way we defined our iteration, it is clear that if $S \subset \omega_1$ is stationary in $W_2[I_{\omega_2}]$, then $\varphi(S)$ holds. 

On the other hand, if $\varphi(S)$ is true and $M_1,M_2$ and $M_3$ are witnesses to the truth of $\varphi(S)$, then, as they see  full patterns, their local patterns must coincide with the patterns in the real world $W_2[I_{\omega_2}]$. But the last Lemma ensures that the patterns of $W_2[I_{\omega_2}]$ characterize stationarity, so the proof is finished.
\end{proof}
As in \cite{Ho}, instead of working in an arbitrary $V$ with a Woodin cardinal, we can work over the canonical inner model with one Woodin cardinal $M_1$. This has the advantage, that we can replace the two parameters $\vec{C}$ and $\vec{T}^0$ by just $\{ \omega_1 \}$. We will not go into any details and just claim that the above proof can be applied over $M_1$, with all the modifications exactly as in \cite{Ho}. We therefore obtain the last theorem of this article.
\begin{theorem}
Let $M_1$ be the canonical inner model with one Woodin cardinal. Then there is a generic extension of $M_1$, in which $\NS$ is saturated, $\Sigma_1(\omega_1)$-definable, and $\MA$ holds.
\end{theorem}

\subsubsection{Open questions}
We end with a couple of natural problems which remained open.
\begin{question}
Assume $\PFA$. Is there a $\Sigma_1$-formula and a set $A \subset \omega_1$ such that 
\[\forall S\in P(\omega_1) (S \text{ is stationary } \Leftrightarrow \varphi(S,A) )? \]
\end{question}

\begin{question}
Assume the existence of a Woodin cardinal. Is there a universe in which
$\BPFA$ holds, $\NS$ is $\bf{\Delta_1}$-definable over $H(\omega_2)$ and $\NS$ is saturated?
\end{question}

\begin{question}
Assume the existence of a reflecting cardinal. Is $\BPFA$ consistent with the non-existence of a $\Sigma_1(\omega_1)$-definition of $\NS$?

\end{question}

\end{document}